\documentclass[twoside,leqno,
]{amsart}
\usepackage{latexsym,amssymb,amsfonts,euscript}
\usepackage[dvips]{graphicx}

\usepackage[pagebackref,colorlinks=true]{hyperref}  

\usepackage{xcolor}
\usepackage{soul}

\newtheorem{thm}{Theorem}[section]
 \newtheorem{cor}[thm]{Corollary}

 \theoremstyle{definition}
 
 \theoremstyle{remark}

 \numberwithin{equation}{section}

\begin{document}

 \markboth{M. Manev and V. Tavkova}  
 {Matrix Lie groups as 3-dimensional almost paracontact almost paracomplex Riemannian manifolds}

%

\newcommand{\ie}{i.e.\ }
\newcommand{\f}{\phi}
\newcommand{\tg}{\tilde{g}}
\newcommand{\n}{\nabla}
\newcommand{\nn}{\tilde{\n}}
\newcommand{\M}{(\mathcal{M},\allowbreak{}\f,\allowbreak{}\xi,\allowbreak{}\eta,g)}
\newcommand{\G}{\mathcal{G}}
\newcommand{\I}{\mathcal{I}}
\newcommand{\W}{\mathcal{W}}
\newcommand{\R}{\mathbb R}
\newcommand{\C}{\mathbb C}
\newcommand{\X}{\mathfrak X}
\newcommand{\F}{\mathcal{F}}
\newcommand{\U}{\mathcal{U}}
\newcommand{\HH}{\mathcal{H}}
\newcommand{\VV}{\mathcal{V}}
\newcommand{\MM}{\mathcal{M}}
\newcommand{\LL}{\mathfrak{L}}
\newcommand{\LLL}{\mathcal{L}}
\newcommand{\hatN}{\widehat{N}}
\newcommand{\tr}{{\rm tr}}
\newcommand{\ta}{\theta}
\newcommand{\shy}{\operatorname{sh}}
\newcommand{\chy}{\operatorname{ch}}
\newcommand{\ze}{\zeta}
\newcommand{\om}{\omega}
\newcommand{\lm}{\lambda}
\newcommand{\ce}{\ceta}
\newcommand{\gm}{\gamma}
\newcommand{\al}{\alpha}
\newcommand{\bt}{\beta}
\newcommand{\Dt}{\Delta}
\newcommand{\sx}{\mathop{\mathfrak{S}}\limits_{x,y,z}}
\newcommand{\D}{\mathrm{d}}
\newcommand{\ddr}{\tfrac{\D}{\D r}}
\newcommand{\ii}{\mathrm{i}}

\newcommand{\im}{\operatorname{im}}
\newcommand{\Span}{\operatorname{span}}
\newcommand{\Id}{\operatorname{Id}}

\newcommand{\thmref}[1]{Theorem~\ref{#1}}
\newcommand{\lemref}[1]{Lemma~\ref{#1}}
\newcommand{\cororref}[1]{Corollary~\ref{#1}}
\newcommand{\propref}[1]{Proposition~\ref{#1}}
\newcommand{\tablref}[1]{Table~\ref{#1}}
\title[Matrix Lie Groups As 3-Dimensional Almost Paracontact \ldots] 
{Matrix Lie Groups as 3-Dimensional Almost Paracontact Almost Paracomplex Riemannian Manifolds}

\author[M. Manev, V. Tavkova]{Mancho Manev and Veselina Tavkova}

\begin{abstract}
Lie groups considered as three-dimensional almost paracontact almost paracomplex Riemannian manifolds are investigated.   
In each basic class of the classification used for the manifolds under consideration, a correspondence is established between the Lie algebra and the explicit matrix representation of its Lie group.
\end{abstract}

\subjclass[2010]{
53C15, 
22E60, 
22E15
}

\keywords{Almost paracontact structure, almost paracomplex structure, Riemannian metric, Lie group, Lie algebra}

\maketitle

\section{Introduction}\label{sec-intro}
 \vglue-10pt
 \indent

In the present paper, we continue the investigations of almost paracontact almost paracomplex Riemannian manifolds. 
In \cite{Sato76}, I. Sato introduced  the concept of (almost) paracontact structure compatible with a Riemannian metric as an analogue of almost contact Riemannian manifold. After that,  a number of authors develop the
differential geometry of these manifolds. The  beginning of the investigations on the  paracontact Riemannian manifolds is given by  \cite{AdatMiya77}, \cite{Sa80}, \cite{Sato77} and \cite{Sato78}.


In \cite{ManSta01}, a classification of almost paracontact Riemannian manifolds of type $(n,n)$ is made, 
taking into account the relevant notion given by Sasaki in \cite{Sa80}. 
They are  $(2n+1)$-dimensional and the induced almost product structure
on the paracontact distribution is traceless, \ie  it is an almost paracomplex structure. 
In  \cite{ManVes}, these manifolds are called \emph{almost paracontact almost paracomplex manifolds}.

In a series of papers, e.g. \cite{AbbGarb, Barb, GriManMek, DobrMek, FCG, FinoGran, VMMolina, HMan, HMan2, Ovando, ZamNak}, the authors consider Lie groups as manifolds equipped with different additional tensor structures and metrics compatible with them. 
Furthermore, in our previous work \cite{ManVes2}, we construct and characterize a
family of 3-dimensional Lie algebras corresponding to Lie groups considered as almost paracontact almost paracomplex
Riemannian manifolds. Curvature properties of these manifolds are studied.


It is known by \cite{Gil} that each representation of a Lie algebra corresponds uniquely to a representation of a simply connected Lie group. This relation is one-to-one. 
Hence, knowledge the representation of a certain Lie algebra settles the issue of the representation of its Lie group.

In the present work, our goal is to find a correspondence between the Lie algebras constructed
in \cite{ManVes2} and explicit matrix representations of their Lie groups for each of the basic classes of the classification used for the manifolds under study.

The paper is organized as follows.  
In Sect.~\ref{sect-mfds}, we recall some necessary facts about the investigated manifolds and related Lie algebras. 
In Sect.~\ref{sect-lie}, we find the explicit correspondence between the Lie algebras determined in all basic classes of the manifolds studied and respective matrix Lie groups.


 \section{Preliminaries}\label{sect-mfds}
 \vglue-10pt
 \indent

\subsection{Almost paracontact almost paracomplex Riemannian manifolds}
 \vglue-10pt
 \indent

Let $(\MM,\f,\xi,\eta,g)$ be an almost paracontact almost paracomplex Riemannian manifold. 
This means that $\MM$ is a $(2n+1)$-dimensional real differentiable manifold equipped with an almost paracontact almost paracomplex structure $(\f,\xi,\eta)$, \ie $\f$ is a fundamental $(1,1)$-tensor field of the tangent bundle $T\MM$ of $\MM$, $\xi$ is a characteristic vector field and $\eta$ is its dual 1-form satisfying the following conditions:
\begin{equation*}\label{str}
\begin{array}{c}
\f^2 = \I - \eta \otimes \xi,\quad \eta(\xi)=1,\quad
\eta\circ\f=0,\quad \f\xi = 0,\quad \tr\f = 0,
\end{array}
\end{equation*}
where $\I$ denotes the identity on $T\MM$. 
Moreover, $g$ is a Riemannian metric 
that is compatible with the structure of the manifold so that the following condition is fulfilled
\begin{equation*}\label{str2}
\begin{array}{c}
g(\f x, \f y) = g(x,y) - \eta(x)\eta(y)
\end{array}
\end{equation*}
for arbitrary $x,y \in T\MM$  \cite{Sato76}, \cite{ManSta01}.

Further $x$, $y$, $z$, $w$ will stand for arbitrary
elements of the Lie algebra $\X(\MM)$ of tangent vector fields on $\MM$ or vectors in the tangent space $T_p\MM$ at $p\in \MM$.

Let us recall that an almost paracomplex structure is a traceless almost product structure $P$, \ie     
 $P^2=\I$,  $P\neq \pm \I$ and $\tr P=0$. 
Because of $\tr P=0$, the eigenvalues $+1$ and $-1$ of $P$ have one and the same multiplicity $n$.


Let $\nabla$ be the Levi-Civita connection generated by $g$. The tensor field $F$ of type (0,3) on $\MM$  is defined by
\begin{equation*}\label{F=nfi}
F(x,y,z)=g\bigl( \left( \nabla_x \f \right)y,z\bigr).
\end{equation*}
The following equalities define 1-forms associated with $F$, known as the Lee forms of $\MM$: 
\begin{equation*}\label{t}
\theta(z)=g^{ij}F(e_i,e_j,z),\quad \theta^*(z)=g^{ij}F(e_i,\f
e_j,z), \quad \omega(z)=F(\xi,\xi,z),
\end{equation*}
where $g^{ij}$ are the components of the inverse matrix of $g$ with respect to a basis $\left\{e_i;\xi\right\}$ $(i=1,2,\dots,2n)$ of  $T_p\MM$ at an arbitrary point $p\in \MM$. 

A classification of almost paracontact almost paracomplex Riemannian manifolds is given in \cite{ManSta01}.
It consists of eleven basic classes $\F_{s}$, $s\in\{1,2,\dots, 11\}$, and each of them is defined by conditions for $F$. 
In \cite{ManVes}, we determine the components $F_{s}$ of $F$ that correspond to each $\F_{s}$. In other words, the manifold $\M$ belongs to $\F_{s}$ 
if and only if the equality $F=F_s$ is satisfied. 
The intersection of the basic classes is the special class $\F_0$
defined by the condition $F=0$, which is equivalent to the covariant constancy of the structure tensors with respect to $\n$, 
\ie $\n\f=\n\xi=\n\eta=\n g=0$.



Let us consider the studied manifold of the lowest dimension, \ie $\dim{\MM}=3$.

Let $\left\{e_0,e_1,e_2\right\}$, where $e_0=\xi,e_1=\f e_2,e_2=\f e_1$, be a \emph{$\f$-basis} of $T_p\MM$.  Thus, it is an orthonormal basis with respect to $g$, \ie $g(e_i,e_j)=\delta_{ij}$ for all $i,j\in\{0,1,2\}$.
In \cite{ManVes}, we determine the components ${F_{ijk}=F(e_i,e_j,e_k)}$, ${\ta_k=\ta(e_k)}$, ${\ta^*_k=\ta^*(e_k)}$ and ${\om_k=\om(e_k)}$ of $F$, $\ta$, $\ta^*$ and $\om$, respectively,  with respect to $\left\{e_0,e_1,e_2\right\}$ as  follows: 
\begin{equation*}\label{t3}
\begin{array}{c}
	\begin{array}{ll}
		\ta_0=F_{110}+F_{220},\quad & \ta_1=F_{111}=-F_{122}=-\ta^*_2,\\[4pt]
		\ta^*_0=F_{120}+F_{210}, \quad &\ta_2=F_{222}=-F_{211}=-\ta^*_1,\\[4pt]
	\end{array}\\
	\begin{array}{lll}
		\om_0=0,  \qquad & \om_1=F_{001},\qquad & \om_2=F_{002}.
	\end{array}
\end{array}
\end{equation*}

Let $x=x^ie_i$, $y=y^ie_i$, $z=z^ie_i$ be arbitrary vectors in $T_p\MM$, $p\in \MM$, decomposed with respect to the $\f$-basis. 
Then, the components $F_s$, $s\in\{1,2,\dots,11\}$, of $F$ on $\M\in\F_s$ have the following form: \cite{ManVes}
\begin{equation}\label{Fi3}
\begin{array}{l}
F_{1}(x,y,z)=\left(x^1\ta_1-x^2\ta_2\right)\left(y^1z^1-y^2z^2\right); \\[4pt]
F_{2}(x,y,z)=F_{3}(x,y,z)=0;
\\
F_{4}(x,y,z)=\frac{\ta_0}{2}\Bigl\{x^1\left(y^0z^1+y^1z^0\right)
+x^2\left(y^0z^2+y^2z^0\right)\bigr\};\\[4pt]
F_{5}(x,y,z)=\frac{\ta^*_0}{2}\bigl\{x^1\left(y^0z^2+y^2z^0\right)
+x^2\left(y^0z^1+y^1z^0\right)\bigr\};\\[4pt]
F_{6}(x,y,z)=F_{7}(x,y,z)=0;\\[4pt]
F_{8}(x,y,z)=\lm\bigl\{x^1\left(y^0z^1+y^1z^0\right)
-x^2\left(y^0z^2+y^2z^0\right)\bigr\},\\[4pt]
\hspace{38pt} \lm=F_{110}=-F_{220}
;\\[4pt]
F_{9}(x,y,z)=\mu\bigl\{x^1\left(y^0z^2+y^2z^0\right)
-x^2\left(y^0z^1+y^1z^0\right)\bigr\},\\[4pt]
\hspace{38pt} \mu=F_{120}=-F_{210}
;\\[4pt]
F_{10}(x,y,z)=\nu x^0\left(y^1z^1-y^2z^2\right),\quad
\nu=F_{011}=-F_{022}
;\\[4pt]
F_{11}(x,y,z)=x^0\bigl\{\om_{1}\left(y^0z^1+y^1z^0\right)
+\om_{2}\left(y^0z^2+y^2z^0\right)\bigr\}.
\end{array}
\end{equation}
Therefore, the basic  classes of the 3-dimensional manifolds of the investigated type are
$\F_1$, $ \F_4$, $\F_5$,  $\F_8$, $\F_9$, $\F_{10}$,  $\F_{11}$, \ie $\F_2$, $\F_3$, $\F_6$, $\F_7$ are restricted to $\F_0$ \cite{ManVes}.

 \indent

\subsection{The Lie algebras corresponding to Lie groups as almost paracontact almost paracomplex Riemannian manifolds}
 \vglue-10pt
 \indent

In this subsection we recall  the necessary results obtained in \cite{ManVes2}. 

Let $\LLL$ be a 3-dimensional real connected Lie group and let
$\mathfrak{l}$ be its Lie algebra with a basis
$\{E_{0},E_{1},E_{2}\}$ of left invariant vector fields. An almost paracontact almost paracomplex structure $(\f,\xi,\eta)$ and a Riemannian metric $g$ are defined as follows:
\begin{equation*}\label{strL}
\begin{array}{c}
\f E_0=0,\quad \f E_1=E_{2},\quad \f E_{2}= E_1,\quad \xi=
E_0,\quad \\[4pt]
\eta(E_0)=1,\quad \eta(E_1)=\eta(E_{2})=0,
\end{array}
\end{equation*}
\begin{equation*}\label{gL}
  g(E_i,E_j)=\delta_{ij},\qquad i,j\in\{0,1,2\}.
\end{equation*}
The resulting manifold $(\LLL,\f,\xi,\eta,g)$ is found to be a 3-dimensional almost paracontact almost paracomplex Riemannian manifold.

The corresponding Lie algebra $\mathfrak{l}$
 is determined as follows
\begin{equation*}\label{lie}
\left[E_{i},E_{j}\right]=C_{ij}^k E_{k}, \quad i, j, k \in \{0,1,2\},
\end{equation*}
where $C_{ij}^k$ are the commutation coefficients.

\begin{thm}[\cite{ManVes2}]\label{thm-Fi-L}
The manifold $(\LLL,\f,\xi,\eta,g)$ belongs to the basic class $\F_s$
($s \in \{1,\allowbreak{}4,5,8,9,10,11\}$) if and only
if the corresponding Lie algebra $\mathfrak{l}$ is determined by
the following commutators:
\begin{equation*}\label{Fi-L}
\begin{array}{llll}
\F_1:\; &[E_0,E_1]=0, \; & [E_0,E_2]=0, \; &   [E_1,E_2]=\al E_1-\bt E_2;
\\[4pt]
\F_4:\; &[E_0,E_1]=\al E_2, \; &   [E_0,E_2]=\al E_1, \; &
[E_1,E_2]=0;
\\[4pt]
\F_5:\; &[E_0,E_1]=\al E_1, \; &   [E_0,E_2]=\al E_2, \; &
[E_1,E_2]=0;
\\[4pt]
\F_8:\; &[E_0,E_1]=\al E_2, \; &   [E_0,E_2]=-\al E_1, \; &
[E_1,E_2]=2\al E_0;
\\[4pt]
\F_9:\; &[E_0,E_1]=\al E_1, \; &   [E_0,E_2]=-\al E_2, \; &
[E_1,E_2]=0;
\\[4pt]
\F_{10}:\; &[E_0,E_1]=-\al E_2, \; &   [E_0,E_2]=\al E_1, \; &
[E_1,E_2]=0;
\\[4pt]
\F_{11}:\; &[E_0,E_1]=\al E_0, \; &   [E_0,E_2]=\bt E_0, \; &
[E_1,E_2]=0,
\end{array}
\end{equation*}
where $\al$, $\bt$ are arbitrary real parameters.
Moreover, the
relations of $\al$ and $\bt$ with the non-zero components
$F_{ijk}$ in the different basic classes $\F_s$ from \eqref{Fi3} are the following:
\begin{equation*}\label{Fi-L-alpha}
\begin{array}{rlrl}
\F_1:& \al=\frac12 \ta_1, \; \bt=-\frac12 \ta_2; \qquad
&\F_4:& \al=\frac12 \ta_0;\\[4pt]
\F_5:& \al=\frac12 \ta^*_0; \qquad
&\F_8:& \al=\lm; \\[4pt]
\F_9:& \al=\mu; \qquad
&\F_{10}:& \al=\frac12 \nu;   \\[4pt]
\F_{11}:& \al=\om_2, \; \bt=\om_1.\qquad & &
\end{array}
\end{equation*}
\end{thm}

Obviously, if $\al$ (and $\bt$ if any) vanish in the corresponding class, then the Lie algebra is Abelian and the manifold belongs to $\F_0$. We further exclude this trivial case from our considerations, \ie we assume that $(\al,\bt)\neq(0,0)$.

Recall that the class of the para-Sasakian paracomplex Riemannian
manifolds is $\F'_4$, which is the subclass of $\F_4$ determined by the condition $\ta(\xi)=-2n$
\cite{ManVes}. Then, \thmref{thm-Fi-L} has the following
\begin{cor}[\cite{ManVes2}]\label{cor:para}
The manifold $(\LLL,\f,\xi,\eta,g)$ is para-Sasakian if and only
if the corresponding Lie algebra $\mathfrak{l}$ is determined by
the following commutators:
\[
[E_0,E_1]=-E_2, \quad   [E_0,E_2]=-E_1, \quad [E_1,E_2]=0.
\]
\end{cor}

 \section{Matrix representation of the 3-dimensional  Lie groups equipped with the structure studied}\label{sect-lie}
 \vglue-10pt
 \indent

Let $(\LLL, \f, \xi, \eta, g)$ be a 3-dimensional almost paracontact almost paracomplex Riemannian manifold, where $\LLL$ is a Lie group with associated Lie algebra $\mathfrak{g}$.
In Theorem \ref{thm-Fi-L}, we determine the Lie algebra by commutators such that the manifold belongs to the class  $\F_s$ ($s \in \{1,4,5,8,9,10,11\}$).

In the following theorem, which is the main theorem in the present work, we obtain an explicit matrix representation of a Lie group  $\G$ isomorphic to the given Lie group  $\LLL$ with the same Lie  algebra $\mathfrak{g}$ when $(\LLL, \f, \xi, \eta, g)$ belongs to each of $\F_s$.
\begin{thm}\label{thm:main}
Let $(\LLL,\f,\xi,\eta,g)$ be an almost paracontact almost paracomplex Riemannian manifold belonging to the basic class $\F_s$
($s \in \{1,4,5,8,9,10,11\}$). Then the compact simply connected Lie group $\G$ isomorphic to $\LLL$, both with one and the same Lie algebra $\mathfrak{g}$, has the following matrix representation
\begin{equation}\label{eA}
  e^A=E+tA+uA^2,
\end{equation}
where $E$ is the identity matrix, $A$ is the matrix representation of the corresponding Lie algebra and $t$, $u$ are real parameters. The matrix form of $A$ as well as the expressions of $t$ and $u$ for each of $\F_s$ are given in Table~\ref{T1}, where $a,b,c$ are arbitrary reals and $\al$, $\bt$ are introduced in Theorem~\ref{thm-Fi-L}.
\end{thm}
\begin{proof}
As it is known from  \cite{Gil}, the commutation coefficients provide a matrix representation $A$ of a Lie algebra. Then, the matrix representation of $\mathfrak{g}$ is the following 
 \begin{equation}\label{A}
A=aM_0+bM_1+cM_2,  \quad a, b, c \in \R,
\end{equation}
where the basic matrices $M_i$ have entries determined by the commutation coefficients of $\mathfrak{g}$ as follows 
\begin{equation}\label{Mij}
(M_i)_j^k=-C_{ij}^k, \quad i, j, k \in \{0,1,2\}.
\end{equation}

\begin{table}
  \caption{The matrix form of $A$ and the expressions of $t$ and $u$ for $\F_s$
	}\label{T1}
\centering
 {
 \footnotesize{
 \begin{tabular}{|l|l|l|}
\hline
$
\F_1:$ & $
        A=\left(
      \begin{array}{ccc}
        0 & 0 & 0 \\
        0 & \al c & {-}\bt c \\
        0 & -\al b & \bt b
      \end{array}
        \right)\quad $ &
$        t=
        \left\{
        \begin{array}{ll}
            \frac{e^{\tr{A}}-1}{\tr{A}}, & \tr{A}\neq 0 \\
                1, & \tr{A}= 0
        \end{array}
        \right.$\\
& $
        \tr{A}=\al c + \bt b $
        &
        $u=0
$
\\
\hline
$
\F_4:$&$
A=\left(
      \begin{array}{ccc}
        0 & \al c & \al b \\
        0 & 0 & -\al a \\
        0 & -\al a & 0 \\
      \end{array}
    \right)\quad $
    &
    $        t=
\left\{
  \begin{array}{ll}
    \frac{\sinh{\sqrt{\frac12 \tr{A^2}}}}{\sqrt{\frac12 \tr{A^2}}}, & \tr{A^2}>0 \\
    1, & \tr{A^2}= 0
  \end{array}
\right.$\\
& $
\tr{A^2}=2\al^2 a^2 $ &
$    u=
\left\{
  \begin{array}{ll}
    \frac{\cosh{\sqrt{\frac12 \tr{A^2}}}-1}{\frac12 \tr{A^2}}, & \tr{A^2}> 0 \\
    0, & \tr{A^2}= 0
  \end{array}
\right.
$
\\
\hline
$
\F_5:$&$
A=\left(
      \begin{array}{ccc}
        0 & \al b & \al c \\
        0 & -\al a & 0 \\
        0 & 0 & -\al a \\
      \end{array}
    \right)\quad $ & $
    t=
\left\{
  \begin{array}{ll}
    \frac{e^{\frac12 \tr{A}}-1}{\frac12 \tr{A}}, & \tr{A}\neq 0 \\
    1, & \tr{A}= 0
  \end{array}
\right.$\\
& $
\tr{A}=-2\al a $ &
$        u=0
$
\\
\hline
$
\F_8:$ & $
A=\left(
      \begin{array}{ccc}
        0 & -\al c & \al b \\
        2\al c & 0 & -\al a \\
        -2\al b & \al a & 0 \\
      \end{array}
    \right)\;
$ &
$    t= \frac{\sin\sqrt{-\frac12\tr{A^2}}}{\sqrt{-\frac12\tr{A^2}}}, \quad \tr{A^2}< 0
$\\
& $
\begin{array}{l}
\tr{A^2}=-2\al^2(a^2+2b^2+2c^2)
\end{array}
$ & $
    u=\frac{\cos{\sqrt{-\frac12\tr{A^2}}}-1}{\frac12\tr{A^2}}, \quad \tr{A^2}< 0
$
\\
\hline
$
\F_9:$ & $
A=\left(
      \begin{array}{ccc}
        0 & \al b & -\al c \\
        0 & -\al a & 0 \\
        0 & 0 & \al a \\
      \end{array}
    \right)\quad $ & $
    t=
\left\{
  \begin{array}{ll}
    \frac{\sinh{\sqrt{\frac12 \tr{A^2}}}}{\sqrt{\frac12 \tr{A^2}}}, & \tr{A^2}> 0 \\
    1, & \tr{A^2}= 0
  \end{array}
\right.$\\
& $
\tr{A^2}=2\al^2 a^2 $ & $
    u=
\left\{
  \begin{array}{ll}
   \frac{\cosh{\sqrt{\frac12 \tr{A^2}}}-1}{\frac12 \tr{A^2}}, & \tr{A^2}> 0 \\
    0,                                                  & \tr{A^2}= 0 \\
  \end{array}
\right.
$
\\
\hline
$
\F_{10}:$ & $
A=\left(
      \begin{array}{ccc}
        0 & \al c & -\al b \\
        0 & 0 & -\al a \\
        0 & -\al a & 0 \\
      \end{array}
    \right)\quad $ & $
    t=
\left\{
  \begin{array}{ll}
     \frac{\sinh{\sqrt{\frac12 \tr{A^2}}}}{\sqrt{\frac12 \tr{A^2}}}, & \tr{A^2}> 0 \\
    1, & \tr{A^2}= 0
  \end{array}
\right.$\\
& $
\tr{A^2}=2\al^2 a^2 $ & $
    u=
\left\{
  \begin{array}{ll}
     \frac{\cosh{\sqrt{\frac12 \tr{A^2}}}-1}{\frac12 \tr{A^2}}, & \tr{A^2}> 0 \\
    0,                                                  & \tr{A^2}= 0 \\
  \end{array}
\right.
$
\\
\hline
$
\F_{11}:$ & $
A=\left(
      \begin{array}{ccc}
        \al b + \bt c & 0 & 0 \\
        -\al a & 0 & 0 \\
        -\bt a & 0 & 0 \\
      \end{array}
    \right)\quad $ &
$    t=
\left\{
  \begin{array}{ll}
    \frac{e^{\tr{A}}-1}{\tr{A}}, & \tr{A}\neq 0 \\
    1,                        & \tr{A}= 0 \\
  \end{array}
\right.$\\
& $
\tr{A}=\al b+\bt c $ & $
    u=0
$\\
\hline
 \end{tabular}
 \/}
 }
\end{table}

\textbf{\emph{The class $\F_1$.}}
%
Firstly, let $(\LLL, \f, \xi, \eta, g)$ belong to  $\F_1$. In this case, the corresponding Lie algebra  $\mathfrak{g}_1$, according to  Theorem~\ref{thm-Fi-L},  is determined by the following way:
\begin{equation*}\label{com1}
[E_0,E_1]=[E_0,E_2]=0,\quad [E_1,E_2]=\al E_1-\bt E_2,
\end{equation*}
where $\al=\frac12 \ta_1$, $\bt=-\frac12 \ta_2$.
Therefore, 
the nonzero commutation coefficients are: 
\begin{equation}\label{Cij}
C_{12}^1=-C_{21}^1=\al,\quad C_{12}^2=-C_{21}^2={-\bt}.
\end{equation}



Because of  \eqref{Mij} and \eqref{Cij}, we have
\[
M_0=\left(
      \begin{array}{ccc}
        0 & 0 & 0 \\
        0 & 0 & 0 \\
        0 & 0 & 0 \\
      \end{array}
    \right),\quad
M_1=\left(
      \begin{array}{ccc}
        0 & 0 & 0 \\
        0 & 0 & 0 \\
        0 & -\al & {\bt}\\
      \end{array}
    \right),\quad
M_2=\left(
      \begin{array}{ccc}
        0 & 0 & 0 \\
        0 & \al & {-\bt} \\
        0 & 0 & 0 \\
      \end{array}
    \right).
\]

We have that $(b,c)\neq (0,0)$ is true, otherwise $A$ is a zero matrix and $\mathfrak{g}$ is Abelian. 
Then, using \eqref{A}, we obtain  the matrix representation $A$ of the considered Lie algebra $\mathfrak{g}_1$ given in Table~\ref{T1}. Therefore, the characteristic polynomial of $A$ has the form: 
\[
P_A(\lm)= \lm^2{(\lm-\al c - \bt b )}
\] 
and its eigenvalues $\lm_i$ ($i={1,2,3}$) are the following: 
\[
\lm_1=\lm_2=0, \qquad \lm_3=\al c + \bt b.
\]
We then obtain the corresponding linearly independent eigenvectors  $p_i$ ($i={1,2,3}$):
\[
p_1(1,0,0)^{\intercal}, \qquad p_2(0,\bt,{\al})^{\intercal}, \qquad p_3(0,-c,b)^{\intercal},
\]
using the notation $^\intercal$ for matrix transpose.
The vectors $p_i$  determine the following matrix:
\begin{equation}\label{matrixP-I}
P=\left(
      \begin{array}{ccc}
        1 & 0 & 0 \\
        0 & \bt & -c \\
        0 & {\al} & b
      \end{array}
        \right).
\end{equation}
Using the matrix $A$ for $\F_1$  in Table~\ref{T1} and \eqref{matrixP-I}, we obtain  $\Dt={\det P}=\tr A$, where we  denote $\Dt:=\al c + \bt b $.

Now, let us consider the first case when $\tr A  \neq 0$ holds, \ie $\Dt \neq 0$ and $\det P \neq 0$.  Then, we obtain the inverse matrix of $P$ as follows:
\[
P^{-1}=\frac{1}{\Dt}\left(
      \begin{array}{ccc}
        1 & 0 & 0 \\
        0 & {b} & {c} \\
        0 & -\al & {\bt}
      \end{array}
        \right).
\]

It is well known the formula 
\begin{equation}\label{J}
e^A=Pe^JP^{-1},	
\end{equation}
where the Jordan matrix $J$ is the diagonal matrix $J=\mathrm{diag}(\lm_1, \lm_2, \lm_3)$. Therefore, the matrix representation of the corresponding Lie group $\G_1$ of the considered Lie algebra  $\mathfrak{g}_1$  in the first case is the following:
\[
\G_1=\left\{
\left(
      \begin{array}{ccc}
        1 & 0 & 0 \\
        0 & 1+\al c t
          & -\bt c t \\
        0 & -\al b t
          & 1 +\bt b t\\
      \end{array}
\right)
\left|
\;
                t=\frac{e^{\Dt}-1}{\Dt},\; \Dt\neq 0
\right.
    \right\}.
\]
This result can be written as
\begin{equation}\label{G-I}
\G_1: \quad 
e^A=E+tA, \quad t=\frac{e^{\tr A}-1}{\tr A},\quad  \tr A\neq 0.
\end{equation}

Let us consider the second case when $\tr A=0$, \ie $\Dt=0$ and $\det P =0$. Then the matrix $P$  is non-invertible and therefore $A$ is nilpotent with some nilpotency index $q$ and $e^A$ can be expressed as follows
\begin{equation*}\label{eAq}
e^A=E+A+\frac{A^2}{2!}+\frac{A^3}{3!}+\cdots+\frac{A^{q-1}}{(q-1)!}.
\end{equation*}
Using  the form of $A$ for $\F_1$ in Table~\ref{T1} and $\Dt=0$, we obtain $A^2$ is a zero matrix, \ie $q=2$. Therefore, in this case we get the matrix representation of the Lie group $\G_1$ for $\mathfrak{g}_1$ in the following way:
 \[
\G_1=\left\{
\left(
      \begin{array}{ccc}
        1 & 0 & 0 \\
        0 & 1+\al c
          & -\bt c \\
        0 & -\al b
          & 1 + \bt b \\
      \end{array}
    \right)
\Bigl|
\;
    \Dt=0
    \right\},
\]
which can be written as 
\begin{equation}\label{G-II}
\G_1:\quad 
e^A=E+A, \quad
    \tr A= 0
		.
\end{equation}

Generalizing \eqref{G-I} and \eqref{G-II}, we get the matrix representation \eqref{eA} of the matrix Lie group $\G_1$, where $A$, $t$ and $u$ are given in Table~\ref{T1} for $(\LLL, \f, \xi, \eta, g)\in \F_1$. 

\textbf{\emph{The classes $\F_5$ and $\F_{11}$.}}
When we consider the cases of $\F_5$ and $\F_{11}$, we notice that  $\tr A$ can be non-zero there, just as for $\F_1$. The results in Table~\ref{T1} for these two classes are obtained in the same way as for $\F_1$.

%
\textbf{\emph{The class $\F_4$.}} 
Now, let us consider $(\LLL, \f, \xi, \eta, g)\in\F_4$. According to Theorem~\ref{thm-Fi-L}, the corresponding Lie algebra  $\mathfrak{g}_4$ is determined by the following way
\begin{equation}\label{C4}
[E_0,E_1]=\al E_2,\quad [E_0,E_2]=\al E_1,\quad [E_1,E_2]=0,
\end{equation}
where $\al=\frac12 \ta_0$.
Bearing in mind \eqref{C4}, the non-zero commutation coefficients are:
\begin{equation}\label{Cij4}
C_{01}^2=-C_{10}^2= C_{02}^1=-C_{20}^1=\al.
\end{equation}

By virtue of \eqref{A}, \eqref{Mij} and \eqref{Cij4}, we obtain the matrix representation $A$ of $\mathfrak{g}_4$ as is given in Table~\ref{T1}. Obviously, we have $\tr A=0$.

We determine the matrix $P$ as in the case of $\F_1$ and obtain
\begin{equation*}\label{matrixP-IV}
P=\left(
      \begin{array}{ccc}
        1 & -b-c & -b+c \\
        0 & a & a \\
        0 & a & -a
      \end{array}
        \right)
\end{equation*}
for $\lm_1=0$, $\lm_2=-\al a$, $\lm_3=\al a$, \ie $J=\mathrm{diag}\{0,-\al a,\al a\}$. 
Therefore, we have $\det P=-2a^2$. Using the form of $A$ in Table~\ref{T1} for $\F_4$ and $\tr A^2 = 2\al^2 a^2$, 
we notice that $P$ is invertible or not depending on $\tr A^2\neq 0$ or $\tr A^2= 0$, respectively.

First, when $\tr A^2$ is non-zero, \ie $\tr A^2>0$ is satisfied, we obtain the inverse matrix of $P$ as follows:
\[
P^{-1}=\frac{1}{2a}\left(
      \begin{array}{ccc}
        2a & 2b & 2c \\
        0 & 1 & 1 \\
        0 & 1 & -1
      \end{array}
        \right).
\]
Then, applying \eqref{J}, the following  matrix representation of the Lie group $\G_4$:
\[
\G_4=\left\{
\left(
      \begin{array}{ccc}
        1 & \frac{b}{a}(1-w)+\frac{c}{a}v & \frac{c}{a}(1-w)+\frac{b}{a}v \\
        0 & w & -v  \\
        0 & -v & w  \\
      \end{array}
    \right)
		\Bigl|
\;
    \; a \neq 0
\right\},
\]
where $v=\sinh(\al a)$ and $w=\cosh(\al a)$. 
This result can be written as
\begin{equation}\label{G-IV-1}
\begin{array}{l}
\G_4: \quad e^A=E+t A+ u A^2,
\quad\\[4pt]
\phantom{\G_4: \quad\ }
t=\frac{\sinh \sqrt{\frac{1}{2}\tr A^2}}{ \sqrt{\frac{1}{2}\tr A^2}},
\quad
u=\frac{\cosh \sqrt{\frac{1}{2}\tr A^2}-1}{\frac{1}{2}\tr A^2},\quad
     \tr A^2> 0.
\end{array}
\end{equation}

Now, we focus on the second case when $\tr A^2$ vanishes, therefore $a=0$ is valid 
and $P$ is not invertible. Then $A$ is nilpotent with a nilpotency index $q=2$. Therefore, we obtain 
\begin{equation}\label{G-IV-2}
\G_4:\quad 
e^A=E+A, \quad
    \tr A^2= 0.
\end{equation}

According to \eqref{G-IV-1} and \eqref{G-IV-2}, the matrix Lie group $\G_4$  has the matrix representation \eqref{eA}, where $A$, $t$ and $u$ are given in Table~\ref{T1}  for $(\LLL, \f, \xi, \eta, g)\in \F_4$.

\textbf{\emph{The classes $\F_9$ and $\F_{10}$.}}
Considering the cases of $\F_9$ and $\F_{10}$, we find that  $\tr A= 0$ and $\tr A^2> 0$ there, just as for $\F_4$. 
The results in Table~\ref{T1} for these two classes are obtained in the same way as for $\F_4$.

\textbf{\emph{The class $\F_8$.}}
%
Finally, let us consider the case when $(\LLL, \f, \xi, \eta, g)$ belongs to $\F_8$.
From Theorem~\ref{thm-Fi-L}, we have the following:
\[
[E_0,E_1]=\al E_2,\quad [E_0,E_2]=-\al E_1,\quad [E_1,E_2]=2\al E_0,
\]
where $\al=\lm$, according to \eqref{Fi3}.
 
In the same way as in the cases for $\F_1$ and $\F_4$, we obtain the matrix form of $A$ in $\mathfrak{g}_8$ as it is shown in  Table~\ref{T1}. 
It implies $\tr A=0$ and 
$\tr A^2 =- 2\al^2 \Dt$, where $\Dt:=a^2+2b^2+2c^2$. 
Since $\Dt$ is positive for $(a,b,c)\neq (0,0,0)$, then
$\tr A^2$ is negative in this non-trivial case. 

Obviously, the characteristic polynomial of $A$ has the form 
$
P_A(\lm)= \lm\bigl(\lm^2+\al^2\Dt\bigr)
$ and
we get the following eigenvalues of $A$:
\begin{equation}\label{lm123}
\lm_1=0, \qquad \lm_2=\rm{i}\al  \sqrt{\Dt}, \qquad \lm_3=-\rm{i}\al  \sqrt{\Dt},
\end{equation}
where $\rm{i}=\sqrt{-1}$.
Next, we obtain the corresponding linearly independent eigenvectors  $p_i$ ($i={1,2,3}$):
\begin{equation*}\label{}
\begin{array}{c}
p_1(a,\, 2b,\, 2c)^\intercal, \quad p_2(-ac-\ii b\sqrt{\Dt},\, -2bc+\ii a\sqrt{\Dt},\, a^2 +2b^2)^{\intercal},\\[4pt]
  p_3(-ac+\ii b\sqrt{\Dt},\, -2bc-\ii a\sqrt{\Dt},\, a^2 +2b^2)^{\intercal}
  \end{array}
\end{equation*}
and they form the following matrix
\begin{equation}\label{matrixP}
P=\left(
      \begin{array}{ccc}
        a & -ac-\ii b\sqrt{\Dt} & -ac+\ii b\sqrt{\Dt} \\
        2b & -2bc+\ii a\sqrt{\Dt}  & -2bc-\ii a\sqrt{\Dt} \\
        2c & a^2 +2b^2 & a^2 +2b^2 
      \end{array}
        \right)
\end{equation}
with $\det P=2\ii (a^2+2b^2)\Dt\sqrt{\Dt}$. 
Therefore, $P$ is invertible (respectively, non-invertible) if and only if $(a,b)\neq (0,0)$ (respectively, $(a,b)= (0,0)$ and $c\neq 0$).

 
 Firstly, let us consider the case when $P$ is invertible, \ie $(a,b)\neq (0,0)$.  Then, we obtain the inverse matrix of $P$ as follows:

\[
P^{-1}=\left(
      \begin{array}{ccc}
        \frac{a}{\Dt} & \frac{b}{\Dt} & \frac{c}{\Dt} \\
        \frac{-ac+\ii b\sqrt{\Dt}}{\Dt(a^2 +2b^2)} & -\frac{2bc+\ii a\sqrt{\Dt}}{2\Dt(a^2 +2b^2)} & \frac{1}{2\Dt} \\
        -\frac{ac+\ii b\sqrt{\Dt}}{\Dt(a^2 +2b^2)} & \frac{-2bc+\ii a\sqrt{\Dt}}{2\Dt(a^2 +2b^2)}  & \frac{1}{2\Dt}
      \end{array}
        \right).
\]

Therefore, we obtain the matrix representation of the  Lie group $\G_8$ for  $\mathfrak{g}_8$  in the following way:
\[
\G_8=\left\{
\left(
      \begin{array}{ccc}
        \al^2 u(a^2 - \Dt)+1& ab\al^2 u-\al t & ac\al^2 u+b\al t \\
         2ab\al^2 u+2c\al t & \al^2 u(2b^2 - \Dt)+1
          & 2bc\al^2 u-a\al t\\
         2ac\al^2 u-2b\al t & 2bc\al^2 u+a\al t
          & \al^2 u(2c^2 - \Dt)+1\\
      \end{array}
   \right)	\Bigl|\
      \begin{array}{l}
                \Dt >  0
      \end{array}
    \right\}.
\]
This result can be written as
\begin{equation}\label{G-8-I}
\begin{array}{l}
\G_8: \quad 
e^A=E+tA+uA^2,
 \quad\\[4pt]
\phantom{\G_8: \quad\ }
 t=\frac{\sin\left(\al\sqrt{\Dt}\right)}{\al \sqrt{\Dt}},\quad  u=\frac{1-\cos\left(\al\sqrt{\Dt}\right)}{\al^2 \Dt},\quad  \Dt> 0.
 \end{array}
\end{equation}

Now, let us consider the case when $\det P=0$ for $P$ in \eqref{matrixP}, \ie $(a,b)= (0,0)$ and $c\neq 0$. 
In this case 
we specialize the form of $A$ and obtain its eigenvectors $p_i$ $(i = 1,2,3)$ corresponding to its eigenvalues $\lm_i$ 
in \eqref{lm123}, where $\Dt$ is specialized as $\Dt=2c^2$. Then, the consequent matrix $P$ has the following form
\begin{equation*}\label{matrixP-II}
P=\left(
      \begin{array}{ccc}
        0 & \frac{\ii \sqrt{2}}{2} & -\frac{\ii \sqrt{2}}{2} \\
        0 & 1  & 1 \\
        1 &0 & 0
      \end{array}
        \right)
\end{equation*}
with $\det P=\ii\sqrt{2}$. 
Obviously,  $P$ is invertible now and then its inverse matrix is the following
\[
P^{-1}=\left(
      \begin{array}{ccc}
        0 & 0 & 1 \\
        -\frac{\ii}{\sqrt{2}} & \frac{1}{2} & 0 \\
        \frac{\ii}{\sqrt{2}} & \frac{1}{2}  & 0
      \end{array}
        \right).
\]
Thus, using formula \eqref{J}, the matrix representation of the Lie group $\G_8$ 
in this case is the following: 
\[
\G_8=\left\{
\left(
      \begin{array}{ccc}
        1-\al^2 c^2 u&-\al c t & 0\\
         2\al c t &    1-\al^2 c^2 u
          & 0\\
         0 & 0
          & 1\\
      \end{array}
   \right)	\Bigl|\;
      \begin{array}{ll}
               (a,b)= (0,0),\;
                 c\neq 0
      \end{array}
    \right\},
\]
which can be written as 
\begin{equation*}\label{G-8-II}
\begin{array}{l}
\G_8: \quad 
e^A=E+tA+uA^2,
 \quad\\[4pt]
\phantom{\G_8: \quad\ }
 t=\frac{\sin\left(\al|c|\sqrt{2}\right)}{\al |c|\sqrt{2}},\quad  u=\frac{1-\cos\left(\al|c|\sqrt{2}\right)}{2\al^2 c^2},
 \end{array}
\end{equation*}
which coincides with \eqref{G-8-I} in the special case of $\Dt=2c^2$.

Finally, the results in both cases for $(a,b)\neq (0,0)$ and $(a,b)= (0,0)$, $c\neq 0$ 
can be combined as it is shown in Table~\ref{T1}
for $(\LLL, \f, \xi, \eta, g)\in \F_8$. 

The latter completes the proof of the theorem.
\end{proof}

 Bearing in mind \cororref{cor:para} and \thmref{thm:main}, we obtain immediately the following 
 \begin{cor}
If $(\LLL,\f,\xi,\eta,g)$ is para-Sasakian, 
then the compact simply connected Lie group $\G$ isomorphic to $\LLL$, both with one and the same Lie algebra, has the form \eqref{eA}, \ie $e^A=E+tA+uA^2$, where for $a,b,c\in\R$ we have
\[
A=\left(
      \begin{array}{ccc}
        0 & -c & - b \\
        0 & 0 & a \\
        0 &  a & 0 \\
      \end{array}
    \right), 
\quad
t=
\left\{
  \begin{array}{ll}
    \frac{\sinh |a|}{|a|}, & a\neq0 \\
    1, &  a= 0
  \end{array}
\right.,
\quad   
u=
\left\{
  \begin{array}{ll}
    \frac{\cosh |a|-1}{|a|}, & a\neq 0 \\
    0, & a= 0
  \end{array}
\right.
.
\]
\end{cor}

 \indent

\subsection*{Acknowledgment}
The authors were supported by project MU19-FMI-020 and FP19-FMI-002 of the Scientific Research Fund,
University of Plovdiv Paisii Hilendarski, Bulgaria.

 \bigskip

{\small\rm\baselineskip=10pt
 \baselineskip=10pt
 \qquad Mancho Manev \par
 \qquad University of Plovdiv Paisii Hilendarski, Faculty of Mathematics and
Informatics \par
 \qquad Department of Algebra and Geometry \par
 \qquad 24 Tzar Asen St, 4000 Plovdiv,
Bulgaria \par
 \qquad \& \par
\qquad  Medical University of Plovdiv, 
Faculty of Public Health \par
 \qquad Department of Medical Informatics, Biostatistics and E-Learning \par
 \qquad 15A 	Vasil Aprilov Blvd, 4002 Plovdiv,
Bulgaria \par
 \qquad {\tt mmanev@uni-plovdiv.bg}

 \bigskip \smallskip

 \qquad Veselina Tavkova \par
 \qquad University of Plovdiv Paisii Hilendarski,
Faculty of Mathematics and Informatics \par
 \qquad Department of Algebra and Geometry\par
 \qquad 24 Tzar Asen St, 4000 Plovdiv,
Bulgaria\par
 \qquad {\tt vtavkova@uni-plovdiv.bg}
 }


\bigskip



\begin{thebibliography}{99}

\bibitem{AbbGarb}
\textsc{E. Abbena {\rm and} S. Garbiero}:
\emph{Almost Hermitian homogeneous manifolds and Lie groups}.
Nihonkai Math. J. {\bf 4} (1993), 1--15.

\bibitem{AdatMiya77}
\textsc{T. Adati {\rm and} T. Miyazawa}: 
\emph{On paracontact Riemannian manifolds}. 
TRU Math.  {\bf 13} (1977), 27--39.

\bibitem{Barb}
\textsc{M.\,L. Barberis}:
\emph{Hypercomplex  structures
on four-dimensional Lie groups}.
Proc. Amer. Math. Soc.  {\bf 128} (4) (1997), 1043--1054.


\bibitem{Gil}
\textsc{R. Gilmore}: \textit{Lie Groups, Lie Algebras and Some of Their Applications}.
John Wiley \& Sons, Inc., New York, 1974.

\bibitem{GriManMek}
\textsc{K. Gribachev, M. Manev {\rm and} D. Mekerov}:
\emph{A Lie group as a 4-dimensional quasi-K\"ahler manifold with Norden metric}.
JP J. Geom. Topol. {\bf 6} (1) (2006), 55--68.

\bibitem{DobrMek}
\textsc{D. Gribacheva {\rm and} D. Mekerov}:
\emph{Canonical connection on a class of Riemannian almost product manifolds}.
J. Geom.  {\bf 102}, 53--71  (2011). https://doi.org/10.1007/s00022-011-0098-7


\bibitem{FCG}
\textsc{E.\,A.  Fern\'{a}ndez-Culma {\rm and} Y. Godoy}:
\emph{Anti-K\"ahlerian geometry on Lie groups}.
Math. Phys. Anal. Geom. {\bf 21}, 8 (2018). https://doi.org/10.1007/s11040-018-9266-4

\bibitem{FinoGran}
\textsc{A. Fino {\rm and} G. Grancharov}: 
\emph{Properties of manifolds with skew-symmetric torsion and special holonomy}. 
Adv. Math. {\bf 189} (2) (2004),  439--500.

\bibitem{VMMolina}
\textsc{V. Mart\'in-Molina}: 
\emph{Paracontact metric manifolds without a contact metric counterpart}.
Taiwanese J. Math. {\bf 19} (1) (2015), 175--191.


%


\bibitem{HMan}
\textsc{H. Manev}: 
\emph{Matrix Lie groups as 3-dimensional almost contact B-metric manifolds}. 
Facta Univ. Ser. Math. Inform. {\bf 30} (3) (2015), 341--351.


\bibitem{HMan2}
\textsc{H. Manev}: 
\emph{Matrix Lie groups as 4-dimensional hypercomplex manifolds with Hermitian-Norden metrics}. 
(2019), 	arXiv:1903.08971 [math.DG]


\bibitem{ManVes}
\textsc{M. Manev  {\rm and} V. Tavkova}: 
\emph{On almost paracontact almost paracomplex Riemannian manifolds}. 
Facta Univ. Ser. Math. Inform. {\bf 33} (5) (2018), 637--657. https://doi.org/10.22190/FUMI\allowbreak{}1805637M


\bibitem{ManVes2}
\textsc{M. Manev  {\rm and} V. Tavkova}: 
\emph{Lie groups as 3-dimensional almost paracontact almost paracomplex Riemannian manifolds}. 
J. Geom. {\bf 110}, 43 (2019). https://doi.org/10.1007/s00022-019-0499-6



\bibitem{ManSta01}
\textsc{M. Manev {\rm and} M. Staikova}: 
\emph{On almost paracontact Riemannian manifolds of type $(n,n)$}. 
J. Geom. {\bf 72}, 108--114 (2001). https://doi.org/10.1007/s00022-001-8572-2


%

\bibitem{Sa80}
\textsc{S. Sasaki}: 
\emph{On paracontact Riemannian manifolds}.
TRU Math. {\bf 16} (1980), 75--86.


\bibitem{Sato76}
\textsc{I. Sat\={o}}: 
\emph{On a structure similar to the almost contact structure}.
Tensor (N.S.) {\bf 30} (1976), 219--224.


\bibitem{Sato77}
\textsc{I. Sat\={o}}: 
\emph{On a structure similar to almost contact structure II}.
Tensor (N.S.) {\bf 31} (1977), 199--205.


\bibitem{Sato78}
\textsc{I. Sat\={o}}: 
\emph{On a Riemannian manifold admitting a certain vector field}.
Kodai Math. Sem. Rep. {\bf 29} (1978), 250--260.

\bibitem{Ovando}
\textsc{G. Ovando}: 
\emph{ Invariant complex structures on solvable real Lie groups}. 
Manuscripta Math. \textbf{103} (2000), 19--30.

\bibitem{ZamNak}
\textsc{S. Zamkovoy {\rm and} G. Nakova}: 
\emph{The decomposition of almost paracontact metric man\-i\-folds in eleven classes revisited}.
J. Geom. {\bf 109}, 18 (2018). https://\allowbreak{}doi.org/10.1007/s00022-018-0423-5


\end{thebibliography}
 \end{document}